\documentclass[11pt, a4paper]{amsart}

\usepackage{latexsym}
\usepackage{amssymb}
\usepackage{amsmath}
\usepackage{amsthm}
\usepackage{amsfonts}
\usepackage{color}
\usepackage{enumerate}

\usepackage{pictexwd,dcpic}

\usepackage{graphicx}

\usepackage{comment}

\newtheorem{theorem}{Theorem}[section]
\newtheorem{proposition}[theorem]{Proposition}
\newtheorem{corollary}[theorem]{Corollary}
\newtheorem{lemma}[theorem]{Lemma}

\theoremstyle{definition}

\theoremstyle{remark}
\newtheorem{remark}[theorem]{Remark}

\numberwithin{equation}{section}

\numberwithin{equation}{section}

\newcommand{\scal}[1]{\langle #1 \rangle}
\newcommand{\virg}[1]{``#1''}

\newcommand{\RR}{\mathbb{R}}
\newcommand{\CC}{\mathbb{C}}

\newcommand{\KK}{\mathbb{K}}

\DeclareMathOperator{\dist}{dist}

\DeclareMathOperator{\Tub}{Tub}

\newcommand{\sphere}{\mathrm{\mathbb{S}}}

\newcommand{\tub}{\ensuremath{\mathrm{Tub}}}

\newcommand{\F}{\ensuremath{\mathcal{F}}}

\newcommand{\DR}{\ensuremath{r}}

\newcommand{\FD}{\ensuremath{ f} }

\newcommand{\fol}{\mathcal{F}}

\newcommand{\Sym}{\mathrm{Sym}}

\def\ol{\overline}

\def\In{\subseteq}

\def\tr{\textrm{tr}}

\def\CC{\mathbb{C}}

\def\HH{\mathbb{H}}

\def\PP{\mathbb{P}}

\def\RR{\mathbb{R}}

\def\mc{\mathcal}

\def\fol{\mc{F}}

\begin{document}



\title[MCF of Singular Riemannian Foliations]{Mean curvature flow of singular Riemannian foliations}


\author[Alexandrino]{Marcos M. Alexandrino}

\author[M. Radeschi]{Marco Radeschi}

\address[Alexandrino]{
Universidade de S\~{a}o Paulo, 
Instituto de Matem\'{a}tica e Estat\'{\i}stica,
 Rua do Mat\~{a}o 1010,05508 090 S\~{a}o Paulo, Brazil}
\email{marcosmalex@yahoo.de, malex@ime.usp.br}

\address[Radeschi]{Mathematisches Institut, WWU M\"unster, Germany.}

\email{mrade\_02@uni-muenster.de}

\thanks{The first author was  supported by CNPq and partially supported by FAPESP. The second author is part of SFB 878: Groups, Geometry \& Actions at M\"unster University.}

\date{\today}


\subjclass[2000]{}
\keywords{Mean curvature flow, isometric actions, foliations}


\begin{abstract}

Given a singular Riemannian foliation on a compact Riemannian manifold, we study the mean curvature flow equation with a regular leaf as initial datum. We prove that if the leaves are compact and the mean curvature vector field is basic, then  any finite time singularity is a singular leaf, and the singularity is of type I.
These results generalize previous results of Liu and Terng,  Pacini and Koike.
In particular our results can be applied to   partitions of  Riemannian manifolds into orbits of actions of compact groups of isometries.

 \end{abstract}

\maketitle




\section{Introduction}

Given a Riemannian manifold $M$ and an immersion $\varphi:L_0\to M$,  a smooth  family of immersions $\varphi_t: L_{0}\to M$,
$t\in[0,T)$  is called a solution of the 
\emph{mean curvature flow} (MCF for short) if $\varphi_t$ satisfies the evolution equation
$$\frac{d}{d t} \varphi_{t}= H(t),\qquad  $$
where $H(t)$ is the mean curvature of $L(t):=\varphi_{t}(L_{0})$. We say that the MCF $\varphi_t$ has \emph{initial datum} $L_{0}$. By abuse of notation, we will often identify $\varphi_t$ with its image $L(t)$, and we will talk about the MCF flow $L(t)$.

In  \cite{Liu-Terng} Liu and Terng studied the mean curvature flow equation in spheres and Euclidean spaces with an \emph{isoparametric submanifold} as initial datum  and they proved, among other things, that such an evolution moves through isoparametric submanifolds up to the (finite time) singularity. Later on, Koike \cite{Koike} generalized Liu and Terng's results to the mean curvature flow on compact symmetric spaces, with isoparametric submanifolds with flat sections as initial datum.

Given an isoparametric submanifold $L$, one can partition the ambient manifold into the submanifolds \virg{parallel} to $L$, which are all isoparametric unless they lie in the focal set of $L$, in which case they have lower dimension (and they are called \emph{focal submanifolds}). Such a partition is a special example of a \emph{singular Riemannian foliation} 
 i.e., a foliation where every geodesic starting perpendicular to a leaf, stay perpendicular to all the leaves it meets (cf.  \cite[page 189]{Molino}). The results of Terng-Liu and Koike can then be restated by saying that given an isoparametric submanifold $L$ of a sphere, Euclidean space or compact symmetric space, the MCF evolution with $L$ as initial datum moves through the leaves of the foliation induced by $L$.

Singular Riemannian foliations induced by an isoparametric submanifold (also called \emph{isoparametric foliations}) are characterized by the following two properties:
\begin{enumerate}
\item[i)] The mean curvature form is \emph{basic} (cf. Section \ref{S:preliminaries}).
\item[ii)] The distribution orthogonal to the regular leaves (i.e., the leaves with maximal dimension) is integrable.
\end{enumerate}
If a singular Riemannian foliation satisfies the former condition it is called \emph{generalized isoparametric}.
In this paper, we generalize the results of Terng-Liu and Koike to the class of generalized isoparametric foliations on compact Riemannian manifolds.

Despite the name, generalized isoparametric foliations are much more general than isoparametric ones.
For example, the following foliations are generalized isoparametric:
\begin{enumerate}
\item Any isometric group action of a connected Lie group $G$ on a Riemannian manifold $M$ induces a singular Riemannian foliation $(M,\fol)$ given by the orbits of $G$ (\emph{homogeneous foliation}) which is generalized isoparametric. By comparison, isoparametric foliations only appear when the group action is polar.
\item Any singular Riemannian foliation in spheres or Euclidean space is generalized isoparametric; cf. \cite[Remark 4.2]{AlexandrinoRadeschi-I}. This includes a new class of foliations, neither homogeneous nor polar, constructed using Clifford algebras; cf. \cite{Radeschi-Clifford}. By contrast, (irreducible) isoparametric foliations in spheres either have cohomogeneity one, or arise from a polar representation \cite{Thorbergsson}.
\item Any singular Riemannian foliation $\fol$ on $\KK\PP^n$, ($\KK=\RR,\CC,\HH$) lifts to a foliation $\fol^*$ on a sphere via the Hopf map $\sphere^{m}\to \KK\PP^n$. Moreover, since Hopf fibrations have totally geodesic fibers the mean curvature vector field of the leaves is preserved under the fibration, and in particular $\fol$ is generalized isoparametric. Among these, the (irreducible) isoparametric foliations in $\CC\PP^n$ were recently classified by Dom\'{i}nguez-V\'{a}zquez \cite{Vazquez} (in this case, there are irreducible inhomogeneous isoparametric foliations if and only if $n+1$ is prime).
\end{enumerate}

Recall that a leaf of a singular Riemannian foliation is called \emph{regular} if it has maximal dimension, and \emph{singular} otherwise.
\begin{theorem}
\label{theorem-main}
Let $(M,\fol)$ be a generalized isoparametric foliation with compact leaves on a compact manifold $M$. Let $L_0\in \fol$ be a regular leaf of $M$ and let $L(t)$ denote the mean curvature flow evolution of $L_0$ with maximal interval of existence $[0,T)$. Then the following statements hold:
\begin{enumerate}[(a)]
\item$L(t)$, $t\in [0,T)$ are regular leaves of $\fol$.
\item If $T<\infty$, then $L(t)$ converges to a singular leaf $L_T$  of $\fol$ and the singularity is of type I, i.e.,
\[
\limsup_{t\to T^-}\|A_t \|_{\infty}^2(T-t)<\infty,
\]
where $\|A_t \|_{\infty}$ is the sup norm of the second fundamental form of $L(t)$.
\end{enumerate}
\end{theorem}
\begin{remark}

The condition that $M$ is compact can be replaced by the more general assumption that the MCF $L(t)$ with initial datum $L(0)=L_{0}$ stays
at a bounded distance from $L_0$. This condition can be verified, for example, for any closed singular Riemannian foliation in Euclidean space.
\end{remark}

The condition of having a finite time singularity holds in several situations. For example, if $L_0$ is  a compact submanifold in Euclidean space, it follows from \cite[Proposition 3.10]{Smoczyk-survey} that the maximal time of existence $T$ of the mean curvature flow is finite. For any generalized isoparametric foliation, we prove in Proposition \ref{P:r_t} that there is a neighbourhood around the singular leaves in which the MCF has always finite time singularities. If moreover the manifold is nonnegatively curved and the foliation is isoparametric, then we have the following stronger result.

\begin{theorem}\label{T:finite-time-sing}
For every isoparametric foliation  on a compact nonnegatively curved space, the MCF with a regular leaf as initial datum has always finite time singularity.
\end{theorem}

If we restrict ourselves to special cases, we obtain strengthenings   of different already known results:
\begin{itemize}
\item Theorems \ref{theorem-main} and \ref{T:finite-time-sing} together generalize the main results of Terng-Liu \cite{Liu-Terng} and Koike \cite{Koike} to the case of isoparametric foliations on compact nonnegatively curved manifolds. Moreover, we also show that the flow has type I singularities independently of the singular leaf to which it converges, thus answering in the positive a question posed in  \cite[Remark]{Liu-Terng}.

\item If $\fol$ is a homogeneous foliation by a Lie group $G$ acting on $M$, Pacini  proved in  \cite[Theorem 2]{Pacini}, among other things, that if a curve of principal orbits $t\to G(t) $ (where $t\in [0,T)$ ) is a solution of a MCF with finite time singularity, the limit is a singular orbit provided that such a limit exists. By our main result, such a limit always exists.
\end{itemize}

Like in the case of orbits of isometric actions in compact manifolds (cf. \cite{Pacini}), it is possible to prove that for generalized isoparametric  foliations the mean curvature of a singular leaf is tangent to the stratum that contains it; see \cite[Prop. 2.10]{Radeschi-Lowdimensional} for the case of singular Riemannian foliations in spheres. Moreover it is possible to check, using for example \cite[Prop. 2.9]{Radeschi-Lowdimensional}, that  the mean curvature  is again basic in each stratum.

Therefore the restriction of a generalized isoparametric foliation to each stratum is again a generalized isoparametric foliation, and we immediately get the following result.

\begin{corollary}
Let $(M,\fol)$ be a generalized isoparametric  foliation with compact leaves on a compact manifold $M$. Let $L$ be a singular leaf and let $\Sigma$ be the stratum containing $L$. Let $L(t)$ be the MCF flow with initial datum $L$ and let $[0,T)$ be the maximal interval of existence of the flow. Then the following statements hold.
\begin{enumerate}[(a)]
\item  $L(t)$ is a singular leaf in $\Sigma$ for every $t\in [0,T)$.
\item If $T<\infty$ then $L_T$ converges to a singular leaf $L_T$ with $\dim L_T<\dim L$, and the singularity is of type I.
\end{enumerate}
\end{corollary}

This paper is divided as follows. In Section \ref{S:preliminaries} we recall the definition and properties of singular Riemannian foliation, while in Section \ref{section-main-proofs} we prove Theorem \ref{theorem-main}. The proofs rely on some, somewhat technical, estimates on the shape operator, which are proved in Section \ref{S:technical}. Section \ref{section-polarfoliation-nonnegative} is devoted to the proof of Theorem \ref{T:finite-time-sing}.

\section{Preliminaries}\label{S:preliminaries}

Given a compact Riemannian manifold $M$, a singular foliation $\fol$ is called \emph{singular Riemannian foliation} if every geodesic that starts perpendicular to a leaf, stays perpendicular to all the leaves it meets; cf.  \cite[page 189]{Molino}. We denote by $\dim \fol$ the maximal dimension of the leaves of $\fol$, and call a leaf $L$ \emph{regular} if $\dim L=\dim \fol$ and \emph{singular} otherwise. The union of regular leaves is open and dense in $M$, it is called \emph{regular stratum} and it is denoted by $M_{reg}$. The union of singular leaves of a fixed dimension is a disjoint union of (possibly non complete) submanifolds, which we call \emph{singular strata} of $\fol$.

If the leaves of $\fol$ are closed, the \emph{leaf space} $M/\fol$ inherits the structure of a Hausdorff metric space, where the distance between two points is defined as the distance between the corresponding leaves. Moreover, the subset $M_{reg}/\fol$ of regular leaves is an orbifold, and the canonical map $\pi:M\to M/\fol$ restricts to a (orbifold) Riemannian submersion $M_{reg}\to M_{reg}/\fol$. A vector field $X$ on $M_{reg}$ is called \emph{basic} if it projects via $\pi_*$ to a well-defined vector field in $M_{reg}/\fol$.

Given a singular Riemannian foliation $(M, \fol)$, we denote by $A$ the shape operator of the leaves of $\fol$.
The \emph{mean curvature} $H$ of $\fol$ at a point $p$ is defined as the mean curvature of the leaf $L_p$ through $p$. Since the regular part of the foliation is defined via a Riemannian submersion, the mean curvature $H$ is smooth on $M_{reg}$. We will see, however, that  the norm of $H$ blows up as it approaches singular strata.

In the regular part of $\fol$ the tangent bundle $TM$ splits as $T\fol\oplus \nu \fol$,
where $T\fol$ is the bundle of the tangent spaces of the leaves in $\fol$, and $\nu\fol$ is its othogonal complement. Moreover, for every regular point $p\in M$ it is possible to define the O'Neill tensor $ON:\nu_p \fol\times \nu_p\fol\to T_p\fol$ as $(x,y)\mapsto ON_x\,y=\frac{1}{ 2}\mathrm{pr}_{T\fol}[X,Y]$ where
$X,Y$ are local vector fields extending $x,y$, and $\mathrm{pr}_{T\fol}$ denotes the orthogonal projection onto $T_p\fol$. If $ON\equiv0$ then the orthogonal distribution $\nu\fol$ is integrable, and the foliation is called \emph{polar}.

A singular Riemannian foliation is called \emph{generalized isoparametric} if the mean curvature $H$  on $M_{reg}$  is basic. If moreover it is also polar, then it is called \emph{isoparametric}.

\subsection{Distinguished tubular neighbourhoods}\label{SS:distinguished tub neighb}
Let $(M,\fol)$ be a singular Riemannian foliation and let $q$ be a point in a (possibly singular) leaf $L$.
We recall (cf. \cite[Theorem 6.1, Proposition 6.5]{Molino}, \cite{AlexBriquetToeben}, \cite{LytchakThorbergsson}) that it is possible to find a neighbourhood $P$ of $q$ in $L$, a cylindrical neighbourhood $O_{\epsilon}=Tub_{\epsilon}(P)$ of $q$ in $M$ and diffeomorphism $\varphi:O_{\epsilon}\to V\In T_qM$ onto a neighbourhood $V$ of the origin, such that:
\begin{enumerate}
\item The image of $\fol|_{O_\epsilon}$ under $\varphi$ is the restriction to $V$ of a singular Riemannian foliation $(T_qM,\fol_q)$.
\item $\varphi(L\cap O_{\epsilon})=T_qL\cap V$.
\item Under the splitting $T_{q}M=T_{q}L\times \nu_{q}L$ the foliation $\fol_{q}$ splits as well as  $T_{q}L\times (\nu_{q}L,\fol_{q}\cap\nu_{q}L)$, i.e., any leaf $L'$ of $\fol_{q}$ is of the form $(L'\cap \nu_{q}L)\times T_{q}L$.
\item For any $\lambda\in (0,1)$, the \emph{homothetic transformation} $h_{\lambda}:O_{\epsilon}\to O_{\epsilon}$, $h_{\lambda}(\exp_pv)=\exp_p\lambda v$ corresponds, under $\varphi$, to the rescaling $(v,x)\mapsto (v,\lambda x)$ for any $(v,x)\in T_{q}L\times \nu_qL$.
\end{enumerate}
We call $O_{\epsilon}$ a \emph{distinguished tubular neighbourhood} of $q$, and denote it simply with $O$ if $\epsilon$ is understood.
The map $\varphi$ is a modification of the normal exponential map and in particular it sends radial geodesics around $L$ to radial  geodesics around $T_{q}L$.


\section{Proof theorem \ref{theorem-main}}
\label{section-main-proofs}

In this section we let $(M,\fol)$ be a generalized isoparametric foliation with closed leaves on a compact manifold. We also fix a regular leaf $L_0$ and  assume that the MCF evolution $L(t)$ with initial datum $L_0$ has maximal interval of existence $[0,T)$, with $T$ finite.

Since the mean curvature of $\fol$ is basic, it projects via $\pi:M\to M/\fol$ to a vector field on the orbifold $M_{reg}/\fol$, and it is immediate to see that $L(t)$ is the preimage of the point $\gamma(t)\in M_{reg}/\fol$, where $\gamma$ is the integral curve of (the projection of) $H$. In particular $L(t)$ is a leaf of $\fol$, and since the dimension of $L(t)$ is constant up to the singular time, we have that $L(t)$ is regular.
We thus proved the following.
\begin{proposition}\label{P:recall}
For any $t\in [0,T)$, $L(t)$ is a regular leaf of $\fol$.
\end{proposition}

The rest of the section is devoted to proving statement (b) of Theorem \ref{theorem-main}.
In section \ref{S:technical} we prove the following result (cf. Corollary \ref{corollary-constant-c2}):

\begin{proposition}
Let $L_q\in \fol$ be a singular leaf. For $\epsilon$ small enough, there exist constants $\delta,c$ such that in the regular part of $\tub_{\epsilon}(L_q)$ the shape operator $A$ of $\fol$ satisfies:
\begin{equation}\label{Lemma-bounds_1}
-(1+\delta)\frac{D}{\DR(x)}-c\leq \tr (A_{\nabla\DR})_x\leq -(1-\delta)\frac{D}{\DR(x)}+c,
\end{equation}
where $D=\dim \fol-\dim L_q$ and $r(x)=\dist(x,L_q)$. Moreover, if $\epsilon'<\epsilon$ then one can choose constants $\delta'\leq\delta$ and $c'\leq c$ associated to $\epsilon'$, and $\lim_{\epsilon\to 0^+}\delta=0$.
\end{proposition}

We can now prove the following:

\begin{proposition}\label{P:r_t}
 Given a singular leaf $L_q$, there exists an $\epsilon=\epsilon(L_q)$ such that if $L(t_{0})$ lies in $\tub_{\epsilon}(L_{q})$ for some $t_0\in [0,T)$ then the following properties hold:

\begin{enumerate}
\item[(a)] For any $t>t_{0}$ the distance $r(t)=\dist(L(t),L_q)$ satisfies
\begin{equation}
\label{eq-0-lemma_r_t}
C_{1}^{2}(t-t_{0})\leq \DR^{2}(t_{0})-\DR^{2}(t)\leq C_{2}^{2}(t-t_{0})
\end{equation} 
where $C_1$ and $C_2$ are positive constants that depend only on $\tub_{\epsilon}(L_{q})$.
\item[(b)] $L(t)\subset \tub_{\epsilon}(L_{q})$ for all $t\in (t_0,T)$, and $T<t_0+\frac{\epsilon}{ C_1^2}$.
\item[(c)] If $L(t)$ converges to $L_q$ at time $T$ then for any $t\in (t_0,T]$,
\begin{equation}
\label{eq-1-lemma_r_t}
C_1\sqrt{T- t}\leq \DR(t)\leq C_2\sqrt{T- t}.
\end{equation} 
\end{enumerate}
\end{proposition}

\begin{proof}
Start with a tubular neighbourhood $\tub_{\epsilon_0}(L_{q})$ in which the distance function $\DR=\dist_{L_{q}}$ is smooth away from $L_q$, and such that Proposition \ref{P:recall} holds for some $\delta$ and $c$.
Fixing $p\in L(t_0)$, consider the curve
$t\to \varphi_t(p)$ such that  $\frac{d}{d t}\varphi_t(p)=H(t)$. Of course $\varphi_t(p)\in L(t)$ for all $t$, and $r(t)=\dist(\varphi_t(p),L_q)$ equals $\dist(L(t),L_q)$ by the equidistance of the leaves.
 Then we have 
\begin{equation*}
r'(t) = \langle \nabla \DR, \varphi'_{t}(p) \rangle= \langle \nabla \DR, H(t)\rangle =\tr(A_{\nabla \DR}).
\end{equation*}

 From eq.\eqref{P:recall},
\begin{equation*}
-(1+\delta)\frac{D}{\DR}-c \leq \tr A_{\nabla \DR} \leq -(1-\delta)\frac{D}{\DR}+c.
\end{equation*}

Now we choose $\epsilon < \min \{\epsilon_{0}, (1-\delta)\frac{D}{c} \}$
and define the constants $C_1$, $C_2$ by
\[
\frac{C_1^2}{2}= (1-\delta)D-c\epsilon,\quad \qquad \frac{C_2^2}{2}=(1+\delta)D+c\epsilon
\]
The above equations  imply
\begin{equation*}
-\frac{C_2^2}{ 2 \DR(t)} \leq \DR \,'(t) \leq -\frac{C_1^2}{2 \DR(t)}
\end{equation*}
 or, equivalently, $-C_2^2\leq (\DR^2(t))'\leq -C_1^2$. 
Integrating this equation we get
\begin{equation}
\label{eq-lemma_r_t-final}
C_{1}^{2}(t-t_{0})\leq \DR^{2}(t_{0})-\DR^{2}(t)\leq C_{2}^{2}(t-t_{0})
\end{equation} 
for $t>t_{0}$ close to $t_{0}$. Since $r^2(t)$ is decreasing, $L(t)$ remains in $\tub_{\epsilon}(L_q)$ for every $t> t_{0}$ and this concludes the proof of (a) and (b).

Statement (c) follows directly from (a). 
 \end{proof}

\begin{remark}
By Proposition \ref{P:r_t} it immediately follows that if $L(t_{0})$ lies in a tubular neighborhood defined as above,  then $T$ must be finite. This does not imply, for a generic $M$, that $T$ is always finite when the initial datum is outside such a tube; see \cite[Example 3]{Pacini}.
 Also note that in the proof of Proposition \ref{P:r_t},  $\epsilon$ has been chosen to be small, more precisely smaller than $(1-\delta)D/c$. 
This was necessary to ensure the existence of the constant $C_1$ (otherwise $C_{1}^{2}$ would be negative). The fact that $\epsilon$ can not
be chosen  bigger (even when it would make sense to talk about tubular neighborhoods) is not a limitation of the proof, but it seems to have a geometrical meaning. 
In fact it is possible to see, e.g. in some examples of isoparametric foliations in Euclidean space, that if the radius of the tube is too big (although the tube is still well defined) then statement (b) of Proposition \ref{P:r_t} is no longer true, i.e., 
the MCF of leaves in a tube of big radius can leave the tube after a finite time.   
\end{remark}

In the next proposition we prove that given a leaf $L_0$, if the MCF $L(t)$ with $L(0)=L_0$ has finite time singularity then it converges to a singular leaf $L_q$ in the Hausdorff sense, i.e., the 
projection of $L(t)$ in the quotient space $M/\F$  converges to the projection of $L_q$. More precisely  we prove  
the first part of statement (b) in Theorem \ref{theorem-main}.

\begin{proposition}
\label{proposition-convergence}
 Let $\F$ be a generalized isoparametric foliation with compact leaves on a complete manifold $M$ and let $L_0$ be a regular leaf. Suppose that the MCF $L(t)$ with initial datum $L(0)=L_0$ stays in a bounded set, and that $L(t)$ has a finite time singularity.
Then $L(t)$ converges in the hausdorff sense
 to some singular leaf $L_{q}$. 
\end{proposition}
\begin{proof}

Since $L(t)$ is contained in a bounded set and $T$ is finite,  it follows 
from \cite[Proposition 9.1.4]{PalaisTerng} that the limit set of $L(t)$ can not be contained in the regular stratum and thus it must be contained in some singular stratum. When $\F$ is homogeneous this also follows from  \cite[Lemma 2,3]{Pacini}.

Now consider a singular leaf $L_{q}$ in the limit set, and take a sequence $\{t_n\}\In [0,T)$ converging to $T$. For any arbitrarily small radius ${\epsilon}$, we 
can find some $t_{\epsilon}$ such that  $L(t_{\epsilon})\in\tub_{{\epsilon}}(L_{q})$  and, by Proposition \ref{P:r_t}, $L(t)\in Tub_{\epsilon}(L_q)$ for every $t\in(t_\epsilon,T)$.
Due to the arbitrariety of ${\epsilon}$ we conclude that $L(t)$ converges to $L_{q}$.

\end{proof}

In what follows, we consider a singular leaf $L_q$ which is the limit of the MCF $L(t)$ with initial datum $L_0$. We want to prove that this singularity is of Type I, this finishing the proof of Theorem \ref{theorem-main}.

Fixing a tubular neighbourhood $\tub_{\epsilon}(L_q)$, we consider the functions $r_{\Sigma}, f: \tub_{\epsilon}(L_q)\to \RR$ such that $r_\Sigma(x)$ is the distance between $L_x$ and the singular strata, and $f(x)$ is the distance between $L_x$ and its focal set. By abuse of notation, we also define $r_{\Sigma}(t)=r_{\Sigma}(L(t))$, $f(t)=f(L(t))$.

In Corollary \ref{C:bound_r_Sigma} we prove the following.
\begin{proposition}
There exists a constant $C$, depending on $\tub_\epsilon(L_q)$, such that for any $t$ close enough to the singular time $T$ we have $r_\Sigma(t)\geq C r(t)$, where $r(t)=\dist(L(t), L_q)$.
\end{proposition}
Together with Proposition \ref{P:r_t}, we have that there is a constant $C_{1}'=C_1C$ such that, close enough to the singular time $T$, one has
\begin{equation}
\label{eq-rSigma-T-t}
r_{\Sigma}(t)\geq C_{1}'\sqrt{T-t}.
\end{equation}

\begin{proposition}\label{P:bound f}
There exists a constant $\sigma\in (0,1)$ such that $f(p)\geq \sigma\,r_{\Sigma}(p)$ for every regular point $p\in M$.
\end{proposition}
\begin{proof}
The functions $r_{\Sigma}$ and $f$ are constant along the leaves of $\fol$, and thus induce functions on the quotient, which we denote with the same letters. By  Lytchak and Thorbergsson \cite{LytchakThorbergsson},
 the first focal point of a leaf $L_p$ corresponds to either a singular leaf, or to a conjugate point in $M/\fol$ of the projection of $L_p$. In the first case, $f(p)=r_{\Sigma}(p)$ and the proposition is proved.

Suppose now that the projection $p^{*}$ of $L_p$ into $M/\F$ has a conjugate point along some geodesic segment $\gamma$ contained in the regular part of $M/\F$.
Clearly $r_{\Sigma}(\gamma(s))\geq r_{\Sigma}(p)-s$.  From  Lytchak and Thorbergsson \cite[Remark 1.1]{LytchakThorbergsson}, 
the supremum $\sup(\sec_{M/\fol}( x^*))$ of the sectional curvatures at a point $x^*$ in $U/\F$ satisfies
\begin{equation}\label{E:LT}
\sup(\sec_{M/\fol}( x^*))\leq \frac{K}{ r_{\Sigma}( x^*)^2},
\end{equation}
for some constant $K$. Together with the previous equations,
\begin{equation}\label{E:sup-curv}
\sec_{M/\fol}(\gamma(s))\leq \frac{K}{ r_{\Sigma}(\gamma(s))^2}\leq  \frac{K}{ (r_{\Sigma}(p)-s)^2}.
\end{equation}
By Rauch's Comparison Theorem, the first conjugate point along $\gamma$ appears after the first conjugate point along a geodesic $\ol{\gamma}$ in a model space with curvature $\kappa(\ol{\gamma}(s))= \frac{K}{ (r_{\Sigma}(p)-s)^2}$.

To compute the conjugate point in such a model, it is enough to find the first positive zero of a solution $h$ to the ODE
\begin{equation}
\label{E:ODE-f}
\left\{\begin{array}{rcl} h''(s)&=& -\frac{K}{ (r_{\Sigma}(p)-s)^2}h(s)\\ h(0)&=&0\end{array}\right.
\end{equation}
If we define $g(s)=h(r_{\Sigma}(p) s)$ then $g$ satisfies the equation
\begin{equation}\label{E:ODE-g}
\left\{\begin{array}{rcl} g''(s)&=&- \frac{K}{ (1-s)^2}g(s)\\ g(0)&=&0\end{array}\right.
\end{equation}
and if $\sigma_0$ is the first zero of $g$ in $(0,1)$, then the first zero of $h$ is $\sigma_0 r_{\Sigma}(p)$ and $f(p)\geq\sigma_0 r_{\Sigma}(p)$.

 On the other hand, if $g$ does not admit any zeroes in $(0,1)$, then $h$ does not admit any zeroes in $(0, r_{\Sigma}(p))$ and therefore the first conjugate point along $\gamma$ appears after 
$r_{\Sigma}(p)$. In either case, we proved that $f(p)\geq\sigma\,r_{\Sigma}(p)$, where
\[
\sigma=\left\{
\begin{array}{cl}
\sigma_0 & \textrm{if there exists a zero $\sigma_0$ of $g$ in }(0,1)\\
1& \textrm{otherwise}
\end{array}
\right.
\] 

Notice that $\sigma$ does not depend on $p$.
\end{proof}

We can now prove statement (b) of Theorem \ref{theorem-main}.

\begin{proposition}\label{P:type-I}
Let $\F$ be a generalized isoparametric foliation  with compact leaves on $M$. 
Let $L(t)$ be a MCF evolution with initial datum $L_0\in\fol$. Assume that the MCF $L(t)$ converges to a singular leaf $L_q$. Then the flow has type I singularity, i.e.,
\begin{equation}\label{equation-TypeI}
\limsup_{t\to T^{-}}\|A(t)\|_{\infty}^2(T-t)<\infty
\end{equation}
where $\|A(t)\|_{\infty}$ is the sup norm of the shape operator of $L(t)$.
\end{proposition} 
\begin{proof}

Fixing $q'\in L_q$, we consider a distinguished tubular neighbourhood $O_{\epsilon}$ around $q'$, with map $\varphi:O_{\epsilon}\to T_{q'}M$ as described in Section \ref{S:preliminaries}. We let $\ol{g}$ denote the pullback of the flat metric in $T_{q'}M$ via $\varphi$. We also denote by $\ol{A}, \ol{f}, \ol{r}_{\Sigma}$, etc., the quantities corresponding to $A,f, r_{\Sigma}$, etc.,  computed using the flat metric $\ol{g}$.

By calculations similar to those behind the proof of equation  \eqref{E:traces-ineq}, we can prove that there exist constants $C_1,C_2$ (that depend only on  $O_{\epsilon}(q')$ and $\varphi$) such that:
\begin{equation}
\label{type1aproximation-eq1}
\|A_t \|_{\infty} \leq C_{1} \| \overline{A}_t \|_{\infty}+ C_{2}.
\end{equation}

On the other hand we claim that
\begin{equation}
\label{eq-bounded-flatmetric-typeI} 
\| \overline{A}_t \|_{\infty}\sqrt{T-t}\leq C_3
\end{equation}
where $C_3$ is a constant that depends only on $O_{\epsilon}(q')$.
In fact, by Lemma \ref{focal} we have $\| \overline{A}_t \|_{\infty}=1/\bar{\FD}(t)$, where again $\bar{\FD}(t)$ is the  distance between  the submanifold $L(t)$ and its
first focal point with respect to the flat metric.
Moreover,  from equation \eqref{eq-rSigma-T-t} we have  $\ol{r}_\Sigma(t) >C r_{\Sigma}(t)>C\sqrt{T-t}$. Applying Proposition \ref{P:bound f} to the flat metric, $\ol{f}(t)>C_3\sqrt{T-t}$ and equation \eqref{eq-bounded-flatmetric-typeI}  follows.

Equations \eqref{eq-bounded-flatmetric-typeI}, \eqref{type1aproximation-eq1} and the compactness of $L_q$ imply \eqref{equation-TypeI}.

\end{proof}

We have already discussed that if $L(t)$ is a MCF with initial datum $L_0\in \fol$ and there is a finite time singularity, then $L(t)$ converges to a singular leaf $L_q$ in the Hausdorff sense. We now show that the convergence is in fact pointwise, i.e. for every $p\in L_0$ the integral curve  $t\to \varphi_t(p)$ of $H$ converges to a point in the singular leaf $L_q$ as $t\to T^-$. 

\begin{proposition}
\label{proposition-convergence-point}

Let $(M,\fol)$ be a generalized isoparametric foliation with compact leaves, and let $L(t)=\phi_t(L_0)$ be the MCF evolution with $L(0)=L_0$ a regular leaf of $\fol$. Assume that $L(t)$ converges to singular leaf $L_q$ in a finite time $T$ and let $p\in L(0)$. Then  $\varphi_t(p)$ converges  to a point of $L_{q}$. 
\end{proposition}
\begin{proof}

Let $\gamma(t)=\varphi_{t}(p)$ be the integral curve of $H$ starting at $p$. By Proposition \ref{P:type-I} there exists a reparametrization $\sigma:[0,1)\to [0,T)$ such that $\beta(s):=\gamma(\sigma(s))$ has $\| \beta'(s)\|<\infty$, consider for example $\sigma(s)=T-T(1-s)^2$.

In what follows we prove that $\beta$ converges to a point of $L_{q}$. 

Fixing some $\epsilon>$, let $\pi: \tub_\epsilon(L_q)\to L_{q}$ be the orthogonal projection. Since  
$\| \beta'(s) \|<\infty$, $\pi\circ\beta:[0,1)\to L_q$ is Lipschitz and thus $\lim_{s\to 1}\pi(\beta(s))=p'$ for some $p'\in L_q$. Since $L(t)$ converges to the leaf $L_{q}$, this concludes the proof.

\end{proof}


\section{Estimates on the shape operator}\label{S:technical}

The goal of this section is to compute bounds for the shape operator of a singular Riemannian foliation, starting with foliations in Euclidean space. We start by recalling the following well known fact.
\begin{lemma}\label{focal}
Given a submanifold $L\In \RR^n$ and a normal vector $x$ to $L$, tangent to the \emph{stratum} $\Sigma_L$, let $\lambda_1,\ldots \lambda_r$ be the eigenvalues of the shape operator $A_x$ counted with multiplicity. Then the focal points of $L$ along the geodesic $\gamma_x(t)=\exp tx$ are at distance $1/\lambda_1,\ldots 1/\lambda_r$.
\end{lemma}

Let $(M, \fol)$ be a singular Riemannian foliation, let $q\in M$ be a singular point, of $\fol$, and let $O_{\epsilon}$ be a distinguished tubular neighbourhood around $q$ (cf. Section \ref{SS:distinguished tub neighb}). Let $g$ denote the restriction to $O_{\epsilon}$ of the metric of $M$ and let $\ol{g}$ denote the pullback of the flat metric on $T_qM$ under $\varphi:O_{\epsilon}\to T_qM$. Let $\nabla,\ol{\nabla}$ denote the Levi-Civita connections of $g$ and $\ol{g}$ respectively, and let $\omega$ denote the connection difference tensor
\[
\omega(X,Y)=\nabla_YX-\ol{\nabla}_YX.
\]
We let $G$ be the symmetric $(1,1)$-tensor such that $g(x,y)=\ol{g}(Gx,y)$ for every $x,y\in T_{q}M|_U$. The splitting $T_qM=T_qL_q\times \nu_qL_q$ induces via $\varphi$ a $\ol{g}$-orthogonal splitting $O_{\epsilon}=P\times S$ such that $P_q=P\times\{s\}$ for some $s\in S$. The submanifolds $S_{q'}=\{q'\}\times S$, $q'\in S$, are called \emph{slices} of $O_{\epsilon}$. Clearly, the slices are flat in the $\ol{g}$ metric and they contain all the $\ol{g}$-orthogonal spaces of the leaves in $O_{\epsilon}$.

Any geometric quantity related to a flat metric will be denoted with a bar, e.g., $\ol{\tr}, \ol{A}$. Given a leaf $L$ of $(\RR^n,\fol)$, denote by $\ol{r}_L$ the distance function from $L$ in the flat metric.

\begin{remark}\label{R:bar-r-is-global}
Given two distinguished tubular neighbourhoods $O_{\epsilon}(q)$, $O_{\epsilon}(q')$ with $q'\in L_q$, the corresponding radial functions $\ol{r}(p)=\ol{\dist}(L_q,p)$ with respect to the two flat metrics agree on the intersection. Therefore, even though the flat metric $\ol{g}$ is only defined locally yet $\ol{r}$ can be uniquely defined on a neighbourhood of the whole leaf $L_q$, and it makes sense to define
\[
\ol{\tub}_{\epsilon}(L_q)=\{p\in M\mid \ol{r}(p)<\epsilon\}.
\]
Even more so, there exists a metric $g_0$ in $\ol{\tub}_\epsilon(L_q)$ such that, for any distinguished neighbourhood $O_{\epsilon}$, $g_0$ has the same transverse metric of $\ol{g}$ (cf. \cite{AlexandrinoDesingularization}). In particular, for any leaf $L\In \ol{\tub}_{\epsilon}(L_q)$ it is possible to define a distance function $\ol{r}_L(p)$ in $\ol{\tub}_{\epsilon}(L_q)$ whose restriction to any distinguished tubular neighbourhood $O_{\epsilon}(q')$, $q'\in L_q$, coincides with $\ol{\dist}(L,p)$ in the flat metric.
\end{remark}

\begin{lemma}
Let $(\RR^n,\fol)$ be a singular Riemannian foliation, and let $L$ be a singular leaf. Then for every $\epsilon_L$ small enough there is a constant $C_L$ such that
\begin{equation}\label{E:bound-A-on-flat}
-\frac{D_x}{ \ol{r}_L(x)}-C_L\leq\left(\ol{\tr}\,\ol{A}_{\ol{\nabla}\ol{r}_L}\right)_x\leq -\frac{D_x}{ \ol{r}_L(x)}+C_L\qquad \forall x\in \Tub_{\epsilon_L}(L)
\end{equation}
with $D_x=\dim L_x -\dim L$.
\end{lemma}
\begin{proof}
Let $\epsilon$ be small enough that the normal exponential map $\exp:\nu^{\leq\epsilon}L\to \Tub_{\epsilon}(L)$ is a diffeomorphism, and let $P: \Tub_{\epsilon}(L)\to L$ denote the metric projection. For every $p\in L$, $S_p=\exp_p(\nu_p^{\leq \epsilon}L)$ is the slice of $\fol$ at $p$. For $\epsilon$ small enough the distribution $V_1(x)=T_xL_x\cap T_xS_p$, $p=P(x)$, has dimension $D_x=\dim L_x-\dim L$ and hence  codimension $\dim L$ in $T_xL_x$. Let $V_2(x)\In T_xL_x$ denote the orthogonal complement of $V_1(x)$. Then the following are satisfied:
\begin{enumerate}
\item $T_xL_x=V_1(x)\oplus V_2(x)$ is an orthogonal decomposition for every $x\in \Tub_{\epsilon}(L)$.
\item $V_2$ is a regular distribution which coincides with $T_pL$ for every $p\in L$.
\item By Lemma \ref{focal},  $V_1(x)$ corresponds to the eigenspace of $\ol{A}_{\ol{\nabla}\ol{r}_L}$ with eigenvalue $-\frac{1}{ \ol{r}_L}$. In particular, $V_2(x)$ consists of a sum of eigenspaces of $\ol{A}_{\ol{\nabla}\ol{r}_L}$.
\end{enumerate}
It follows that $\ol{\tr}\,\ol{A}_{\ol{\nabla}\ol{r}_L}=\ol{\tr}\,\ol{A}_{\ol{\nabla}\ol{r}_L}\big|_{V_1}+ \ol{\tr}\,\ol{A}_{\ol{\nabla}\ol{r}_L}\big|_{V_2}$ and, for every $x=\exp_pv$ in $\Tub_{\epsilon}(L)$,
\begin{align*}
\left(\ol{\tr}\,\ol{A}_{\ol{\nabla}\ol{r}_L}\big|_{V_1(x)}\right)_x=-\frac{D_x}{ \ol{r}_L(x)},\quad\qquad
\left|\left(\ol{\tr}\,\ol{A}_{\ol{\nabla}\ol{r}_L}|_{V_2(x)}\right)_x- \left(\ol{\tr}\,\ol{A}_{v}\right)_p\right|<\delta
\end{align*}

where $\delta=\delta(\epsilon)$ is arbitrarily small. The result follows, by letting $C_L=\sup_{v\perp L, \|v\|=1}\left(\ol{\tr}\,\ol{A}_{v}\right)+\delta$.
\end{proof}

\begin{remark} \label{R:homothetic-invariant} Suppose that every leaf $L$ of $(\RR^n,\fol)$ splits isometrically as $V\times L^{\perp}$, where $V$ is a fixed totally geodesic leaf of $\fol$ and $L^\perp\In V^{\perp}$. The homothetic transformations $h_{\lambda}$ at $V$ act on $\RR^n=V\times V^{\perp}$ by fixing $V$ and rescaling the $V^{\perp}$ factor. In particular,
\[
(h_{\lambda})_*\left(\ol{A}_{\ol{\nabla}\ol{r}_L}\right)=\frac{1}{ \lambda}\ol{A}_{\ol{\nabla}\ol{r}_{\lambda L}}
\]
where $\lambda L=h_{\lambda}(L)$. In particular, if $C_L$ satisfies equation \eqref{E:bound-A-on-flat} on $\Tub_{\epsilon}(L)$, then $C_{\lambda L}=\frac{1}{ \lambda}C_L$ satifies equation \eqref{E:bound-A-on-flat} for $\lambda L$, in $\Tub_{\lambda \epsilon}(\lambda L)$. If we define $c_L=\frac{C_L}{ \ol{r}(L)}$, where $\ol{r}(x)=\dist(x,V)$, then $c_L$ becomes invariant under homothetic transformations ($c_L=c_{\lambda L}$) and equation \eqref{E:bound-A-on-flat} becomes
\begin{equation}\label{E:bound-A-on-flat2}
-\frac{D_x}{ \ol{r}_L(x)}-\frac{c_L}{ \ol{r}(x)}\leq\left(\ol{\tr}\,\ol{A}_{\ol{\nabla}\ol{r}_L}\right)_x\leq -\frac{D_x}{ \ol{r}_L(x)}+\frac{c_L}{ \ol{r}(x)}\quad \forall x\in \Tub_{\epsilon_L}(L).
\end{equation}
Clearly, if $(L,\epsilon_L, c_L)$ satisfy equation \eqref{E:bound-A-on-flat2}, then $(\lambda L, \lambda \epsilon_L, c_L)$ satisfy equation \eqref{E:bound-A-on-flat2} as well, for every $\lambda$.
\end{remark}

The next lemma holds for generic Riemannian metrics.

\begin{lemma}
\label{lemma-constant-c2}
Let $(M,\F)$ be a singular Riemannian foliation with compact leaves on a complete Riemannian manifold and let $L_q$ be a singular leaf. Fix $\epsilon>0$ small enough. Then for any $L$ in $\tub_{\epsilon}(L_q)$, there is a radius $\epsilon_L$ and a constant $k_L$ such that in the regular part of $\tub_{\epsilon_L}(L)$ the following holds
\begin{equation}\label{Lemma-bounds_2}
-(1+\delta)\frac{D_{L}}{\ol{\DR}_{ L}(x)}-\frac{k_L}{ \ol{r}(x)}\leq \tr (A_{\nabla\ol{\DR}_{ L}})_x\leq -(1-\delta)\frac{D_{L}}{\ol{\DR}_{ L}(x)}+\frac{k_L}{ \ol{r}(x)}.
\end{equation}
Here the constant $\delta$ only depends on $L_q$ and $\epsilon$, while $k_L$ is homothety invariant (i.e. $k_L=k_{\lambda L}$).
\end{lemma}

\begin{proof}  Fix a distinguished tubular neighbourhood $O_{\epsilon}$ around some point in $L_q$, and let $\ol{g}$ denote the flat metric. In the following, every overlined geometrical quantity is computed with respect to $\ol{g}$.
Using $\nabla=\ol{\nabla}+\omega$ and $g(x,y)=\ol{g}(Gx,y)$, it is not hard to prove that there are constants $\delta, c$ depending only on $L_q$ and $\epsilon$, with $\lim_{\epsilon\to 0}\delta=0$, such that $g$ and $\ol{g}$ are $\delta$-close in the $C^0$-topology and
\begin{equation}\label{E:traces-ineq}
\left(1-\delta\right)\left|\ol{\tr}\,\ol{A}_{\ol{\nabla}\ol{\DR}_{ L}}\right|-c\leq \left|\tr A_{\nabla \ol{\DR}_{ L}}\right|\leq \left(1+{\delta}\right)\left|\ol{\tr}\,\ol{A}_{\ol{\nabla}\ol{\DR}_{ L}}\right|+c.
\end{equation}
Since the metric $\ol{g}$ splits as in Remark \ref{R:homothetic-invariant}, we obtain that for every $L$ there is a homothety invariant $c_L$ and a small $\epsilon_L$ such that equation \eqref{E:bound-A-on-flat2} applies. Using equation \eqref{E:traces-ineq}, we obtain
\begin{equation}\label{E:ineq-complete}
-(1+\delta)\frac{D_{L}}{\ol{\DR}_{ L}(x)}-(1+\delta)\frac{c_L}{ \ol{r}(x)}-c\leq \tr (A_{\nabla\ol{\DR}_{ L}})_x\leq -(1-\delta)\frac{D_{L}}{\ol{\DR}_{ L}(x)}+(1-\delta)\frac{c_L}{ \ol{r}(x)}+c,
\end{equation}

By setting $k_L=(1+\delta)c_L+\epsilon c$ we obtain the result.

\end{proof}

In the particular case of $L=L_q$, we can choose $\epsilon_L=\epsilon$ and follow the same steps as above, noticing that in this case $\ol{\DR}_{ L}=\DR$ and thus $\nabla\ol{\DR}_{ L}=\nabla\DR$. Moreover, in this case we get $C_L=c_L=0$, thus from equation \eqref{E:ineq-complete} we get the following

\begin{corollary}\label{corollary-constant-c2}
Let $(M,\F)$ be a singular Riemannian foliation with compact leaves on a complete Riemannian manifold and let $L_q$ be a singular leaf. For $\epsilon>0$ small enough, there exist constants $\delta,c$ such that in the regular part of $\tub_{\epsilon}(L_q)$
\begin{equation}\label{corollary-bounds_1}
-(1+\delta)\frac{D}{\DR(x)}-c\leq \tr (A_{\nabla\DR})_x\leq -(1-\delta)\frac{D}{\DR(x)}+c.
\end{equation}
where $D=\dim\fol-\dim L_q$ and $r=\dist_{L_q}$.
\end{corollary}

\begin{remark}
The above corollary implies that  there is no Riemannian metric on $M$, adapted to 
a singular Riemannian foliation $\F$ with compact leaves, 
for which all the leaves of $\F$ are minimal submanifolds; see also Miquel and Wolak \cite{MiguelWolak}.  
\end{remark}

\begin{corollary}\label{C:bound_r_Sigma}

Let $M,\fol, L_q$ be as in Lemma \ref{lemma-constant-c2} and assume that $\F$ is generalized isoparametric.
Let $\tub_{\epsilon}(L_q)$ be a tubular neighbourhood of $L_q$ with radius $\epsilon$ small enough and let $\mathcal{M}$ be the union of the singular leaves in $\tub_{\epsilon}(L_q)$. Then 
there exists a foliated neighbourhood $U$ of $\mathcal{M}\setminus L_q$ with the following two properties:
\begin{enumerate}
\item There exists a constant $C$ such that for any $x\in \tub_{\epsilon}L_q\setminus U$, $\dist(x,\mathcal{M})>C\,\dist(x,L_q)$.
\item for any regular leaf $L_0\in U$, the MCF evolution $L(t)$ with $L(0)=L_0$ does not converge to $L_q$.
\end{enumerate}
\end{corollary}

\begin{proof}
Let $\mathcal{L}$ denote the set of singular leaves in $\ol{\tub_{\epsilon}}(L_q)$, and define
\[
U=\bigcup_{L\in \mathcal{L}}\ol{\tub}_{\epsilon_L}(L).
\]
Here the tubes $\ol{\tub_{\epsilon}}(L_q), \ol{\tub_{\epsilon}}(L)$ are defined using the distance functions $\ol{r}(p)=\ol{\dist}(L_q,p)$ and $\ol{r}_L(p)=\ol{\dist}(L,p)$ (see Remark \ref{R:bar-r-is-global}), while $\epsilon_L$ is some radius satysfying Lemma \ref{lemma-constant-c2} and rescaling linearly under $\ol{g}$-homothetic transformations. In this way, for any distinguished tubular neighbourhood $O_{\epsilon}=P\times S$ around $L_q$, the restriction $U\cap O_{\epsilon}$ has the form
 $P\times\{\text{conical open set in }S\}$. Clearly there is some constant $C'$ such that $\ol{\dist}(x,\mathcal{M})>C'\ol{\DR}_{ L}(x)$ for every $x$ in $O_{\epsilon}$. Since the metrics $g, \ol{g}$ are equivalent, the first statement follows.
 
 In order to prove the second statement, we choose $\epsilon_L<\frac{(1-\delta)k_L}{ D_L}\ol{r}(L)$. Notice that the right hand side of the inequality rescales linearly under $\ol{g}$-homothetic transformations, thus we can still choose $\epsilon_L$ with the same property. Let $L(t)$ be a MCF evolution with initial datum $L_0\In U$. Then $L_0$ belongs to $\tub_{\epsilon_L}(L)$ for some singular leaf $L\In U$. If we define $\ol{r}_L(t)=\ol{r}_L(L(t))$, by Lemma \ref{lemma-constant-c2} we obtain
 \[
 \ol{\DR}_{L}'(t)=\tr A_{\nabla\ol{\DR}_{ L}}< -(1-\delta)\frac{D_{L}}{\ol{\DR}_{ L}(x)}+\frac{k_L}{ \ol{r}(x)}.
 \]
Since $\ol{r}_L<\epsilon_L<\frac{(1-\delta)k_L}{ D_L}\ol{r}(L)$, we obtain $\ol{r}_L'(t)<0$ and therefore $L(t)$ never leaves $\ol{\tub_{\epsilon_L}}(L)$.
\end{proof}


\section{Isoparametric foliations in nonnegative curvature}
\label{section-polarfoliation-nonnegative}

The goal of this section is to prove Theorem \ref{T:finite-time-sing} which we restate here.
\begin{theorem}\label{P:polar-fol}
Let $(M,\fol)$ be a  isoparametric foliation (i.e., polar and generalized isoparametric)   on a compact nonnegatively curved manifold. Then for every non minimal regular leaf $L_0$, the MCF $L(t)$ with initial datum $L_0$ has finite time singularity.
\end{theorem}

We start by proving a few lemmas.

\begin{lemma}\label{L:no-finite-time-singularity}
Let $(M, \fol)$ be a closed, generalized isoparametric, singular Riemannian foliation on a compact manifold.
\begin{enumerate}
\item If $\mathrm{vol}:M_{reg}\to \RR$ denotes the volume function $x\mapsto \mathrm{vol}(L_x)$ then $H=-\nabla (\log\mathrm{vol})$ in $M_{reg}$.
\item Fixing a regular leaf $L_0$, suppose that the MCF $L(t)$ with $L(0)=L_0$ does not have a finite time singularity. Then there exists a sequence of leaves $L_i$ converging to a minimal regular leaf $L'$ in the Hausdorff sense, such that $\mathrm{vol}(L_i)>\mathrm{vol}(L')$.
\end{enumerate}
\end{lemma}
\begin{proof}
1) Let $\omega$ denote the volume form of the regular leaves. By \cite[Prop. 4.1.1]{GromollWalschap}, given a basic vector field $X$ along a regular leaf $L_p$, we obtain
\begin{align*}
X(\mathrm{vol})(p)&=\int_{L_p}\mathcal{L}_X(\omega)\\
&=-\int_{L_p}\scal{X,H}\omega\\
&=-\scal{X,H}\mathrm{vol}(p)
\end{align*}
where the last equality holds because both $X$ and $H$ are basic, and therefore $\scal{X,H}$ is constant along $L_p$. Dividing the equation by $\mathrm{vol}(p)$ we obtain
\[
\scal{X, \nabla (\log\mathrm{vol}) (p)}=X(\log \mathrm{vol})(p)=-\scal{X,H}
\]
hence the result.

2) From Proposition \ref{P:r_t}, there is a neighbourhood of the singular set $U$ such that every MCF entering $U$ has a finite time singularity, and therefore our flow $L(t)$ must lie in $M\setminus U$, which is a relatively compact subset of $M_{reg}$ whose distance to the singular set is positive. Via the projection $\pi:M\to M/\fol$, $L(t)$ is projected to an integral curve of the vector field $\pi_*H$.
Since $(M\setminus U)/\fol$ is relatively compact, there exists a sequence of times $t_i$ going to infinity, such that $\pi(L(t_i))$ converges to some point $\pi(L')\in (\ol{M\setminus U})/\fol$.
Since $\log\mathrm{vol}(L(t))$ is decreasing, $\log\mathrm{vol}(L(t))>\log\mathrm{vol}(L')>c$ for some $c\in \RR$. On the other hand, from the previous result one has
\[
\frac{d}{dt}\Big(\log\mathrm{vol}(L(t))\Big)=H\Big(\log \mathrm{vol}\big(L(t)\big)\Big)=-\|H\|^2
\]
and since $\log\mathrm{vol}(L(t))$ is bounded from below, then (up to taking a subsequence) one has $\|H|_{L(t_i)}\|^2\to 0$. By the continuity of the mean curvature in $M_{reg}$, $H|_{L'}=0$ and therefore $L'$ is minimal. On the other hand, $L'$ is not a local maximum because it is obtained as a Hausdorff limit of leaves with bigger volume.
\end{proof}

\begin{proof}[Proof of Proposition \ref{P:polar-fol}] Suppose that there is a MCF $L(t)$ without a finite time singularity. By Lemma \ref{L:no-finite-time-singularity}, there exists a sequence of leaves $L_i$ converging to a minimal regular leaf $L'$, such that $\mathrm{vol}(L_i)>\mathrm{vol}(L')$. This will provide a contradiction with the following result, which will then finish the proof.

\begin{proposition}
Let $(M,\fol)$ be a polar foliation on a compact nonnegatively curved manifold. Then for every regular minimal leaf $L'$, there exists a tubular neighbourhood $U$ around $L'$ such that, for every leaf $L$ in $U$, $\mathrm{vol}(L)\leq \mathrm{vol}(L')$.
\end{proposition}
\begin{proof}
Fixing a unit-length, basic vector field $X$ along $L'$ and a point $p\in L'$ let $\gamma_X(s)$ denote the geodesic starting at $p$ with initial velocity $X(p)$. We set
\[
\delta(X)=\sup\{s\mid \mathrm{vol}(L_{\gamma_X(s)}\leq \mathrm{vol}(L')\}.
\]
In order to prove the Proposition, it is enough to show that $\delta(X)>c>0$ for some $c$ not depending on $X$.

Let $e_1,\ldots e_n$ be an orthonormal frame of $T_pL'$, let $E_1(s),\ldots E_n(s)\in T_{\gamma_X(s)}L_{\gamma_X(s)}$ be the extension of $e_1,\ldots e_n$ along $\gamma_X(s)$ by (vertical) parallel transport, which allow us to identify the tangent spaces $T_{\gamma_X(s)}L_{\gamma_X(s)}$ with $T_pL'$. Moreover, let $\omega_s(p)=E_1^*(s)\wedge\cdots \wedge E_n^*(s)$ denote the volume forms of $L_{\gamma_X(s)}$ at $\gamma_X(s)$.

The \emph{holonomy map} $f_s:L'\to L_{\gamma_X(s)}$ defined by $f_s(q)=\exp_{q}sX(q)$ is a well defined, smooth diffeomorphism between $L'$ and $L_{\gamma_X(s)}$, whose differential at a point $q$ is given by ${f_s}_*(e_i)=J_i(s)$, where $J_i$ is the unique \emph{holonomy Jacobi field} starting at $q$ with $J_i(0)=e_i$ (cf. \cite[Sect. 1.4]{GromollWalschap} for the definition and properties of holonomy Jacobi fields).

The volume function along $\gamma_X(s)$ then reads
\[
\mathrm{vol}(L_{\gamma_X(s)})=\int_{L_{\gamma_X(s)}}\omega_s=\int_{L'}f^*_s\omega_s=\int_{L'}j_s(q)\omega
\]
where $j_s(q)=\det(J_1(s),\ldots J_n(s))$.
Since the curvature is nonnegative and the foliation is polar, by standard comparison theory (cf. \cite{Eschemburg}), $j_s(q)$ is bounded above by a corresponding function $\ol{j}_s(q)$ in Euclidean space. In other words, let $\ol{S}_q:[0,b]\to \Sym^2(T_qL')$ be the tensor satisfying 
\[
\ol{S}_q'+\ol{S}_q^2=0,\qquad \ol{S}_q(0)=-A_{X(q)}
\]
and let $\ol{j}_s(q)$ be the function such that
\[
\frac{d}{ds}\ol{j}_s(q)=\ol{j}_s(q)\cdot\tr(\ol{S}_q(s)), \qquad \ol{j}_0(q)=1.
\]
Then $j_s(q)\leq \ol{j}_s(q)$. It is easy to compute $\ol{j}_s(q)$:
\[
\ol{j}_s(q)=(-1)^n \left(\det A_{X(q)}\right) \prod_i (s-\lambda_i(q)^{-1})
\]

where $\lambda_1(q),\ldots \lambda_n(q)$ are the eigenvalues of $A_{X(q)}$. Such a function has a local maximum at $0$, where $\ol{j}_0(q)=j_0(q)=1$. Moreover, this is a maximum in the interval $\left[\frac{1}{ \lambda^-(q)},\frac{1}{\lambda^+(q)}\right]$, where $\lambda^-(q)$ is the smallest (negative) eigenvalue of $A_{X(q)}$ and $\lambda^+(q)$ is the biggest (positive) eigenvalue. In particular, if $\lambda^+_X=\max_{q}\lambda^+(q)$, then $j_s(q)\leq \ol{j}_s(q)\leq 1$ for all $q\in L'$ and $s\in \left[0,\frac{1}{\lambda^+_X}\right]$, and therefore
\[
\mathrm{vol}(L_{\gamma_X(s)})=\int_{L'}j_s(q)\omega\leq \int_{L'}\omega=\mathrm{vol}(L')\qquad \forall s\in [0,1/\lambda^+_X].
\]
Therefore, $\delta(X)\geq 1/\lambda^+_X$. By letting $c=1/\|A\|_{\infty}$, we then have $\delta(X)>1/\|A\|_{\infty}>0$ for every $X$.
\end{proof}
\end{proof}


\bibliographystyle{amsplain}

\end{document}